\RequirePackage{ifthen}
\RequirePackage{ifpdf}
\newcommand{\driverOption}{}
\ifthenelse{\boolean{pdf}}{
  \renewcommand{\driverOption}{pdftex}
} { 
  \renewcommand{\driverOption}{dvips}
}

\documentclass[reqno, 11pt, letterpaper, commented, oneside, \driverOption]{amsart}

\newboolean{isCommented}
\usepackage{geometry}
\usepackage{bbm}
\usepackage[final]{graphicx}
\usepackage{upref}
\usepackage{enumerate}
\usepackage{latexsym}
\usepackage{amssymb}
\usepackage[ansinew]{inputenc}
\usepackage[T1]{fontenc}
\usepackage{mathtools}
\usepackage{mycommands}

\ifthenelse{\boolean{pdf}}{
  \usepackage[final]{ps4pdf}
} { 
  \usepackage[inactive]{ps4pdf}
}

\PSforPDF{\usepackage{psfrag}}

\newcommand{\hyperrefDriverOption}{}
\ifthenelse{\boolean{pdf}}{
	\renewcommand{\hyperrefDriverOption}{pdftex}
} { 
	\renewcommand{\hyperrefDriverOption}{hypertex}
}
\usepackage[\hyperrefDriverOption,
  colorlinks = false,
  pdfauthor={Torsten Mütze, Ueli Peter},
  pdftitle={On globally sparse Ramsey graphs}]
  {hyperref}
    
\ifthenelse{\boolean{isCommented}} {
	\newcommand{\TM}[1]{\marginpar{\parbox{6cm}{{\small {\bf TM:} #1}}}}
	\newcommand{\UP}[1]{\marginpar{\parbox{6cm}{{\small {\bf UP:} #1}}}}
} { 
	\newcommand{\TM}[1]{}
	\newcommand{\UP}[1]{}
}

\newtheorem{theorem}{Theorem}
\newtheorem{lemma}[theorem]{Lemma}

\newtheorem{proposition}[theorem]{Proposition}

\theoremstyle{definition}
\newtheorem{definition}[theorem]{Definition}

\theoremstyle{remark}

\long\def\symbolfootnote[#1]#2{\begingroup
\def\thefootnote{\fnsymbol{footnote}}\footnote[#1]{#2}\endgroup}

\ifthenelse{\boolean{isCommented}} {
	\geometry{
	  hmargin={15mm, 70mm},
	  marginparwidth=40mm, 
	  vmargin={25mm, 25mm},
	  headsep=10mm,
	  headheight=5mm,
	  footskip=10mm
	}
} { 
	\geometry{
	  hmargin={25mm, 25mm}, 
	  vmargin={25mm, 25mm},
	  headsep=10mm,
	  headheight=5mm,
	  footskip=10mm
	}
}

\begin{document}

\begin{center}

\LARGE On globally sparse Ramsey graphs
\vspace{2mm}

\Large Torsten Mütze \quad Ueli Peter
\vspace{2mm}

\large
  Institute of Theoretical Computer Science \\
  ETH Zürich, 8092 Zürich, Switzerland \\
  \{{\small\tt muetzet|upeter}\}{\small\tt @inf.ethz.ch}
\vspace{5mm}

\small

\begin{minipage}{0.8\linewidth}
\textsc{Abstract.}
We say that a graph $G$ has the Ramsey property w.r.t.\ some graph $F$ and some integer $r\geq 2$, or $G$ is $(F,r)$-Ramsey for short, if any $r$-coloring of the edges of $G$ contains a monochromatic copy of $F$. R{\"o}dl and Ruci{\'n}ski asked how globally sparse $(F,r)$-Ramsey graphs $G$ can possibly be, where the density of $G$ is measured by the subgraph $H\seq G$ with the highest average degree. So far, this so-called Ramsey density is known only for cliques and some trivial graphs $F$. In this work we determine the Ramsey density up to some small error terms for several cases when $F$ is a complete bipartite graph, a cycle or a path, and $r\geq 2$ colors are available.
\end{minipage}

\end{center}

\vspace{5mm}

\section{Introduction}

Ramsey's celebrated theorem~\cite{Ramsey01011930} states that for any integers $r$ and $\ell$, \emph{any} $r$-coloring of the edges of a large enough complete graph contains a monochromatic clique on $\ell$ vertices, i.e., a clique whose edges all receive the same color. In this context we say that a graph $G$ \emph{has the Ramsey property w.r.t.\ some graph $F$ and some integer $r\geq 2$}, or $G$ is \emph{$(F,r)$-Ramsey} for short, if any $r$-coloring of the edges of $G$ contains a monochromatic copy of $F$.
While Ramsey's theorem seems to rely on the fact that a large complete graph is very dense, Folkman~\cite{MR0268080} proved that there are graphs that are Ramsey with respect to $K_\ell$ and $r=2$ colors which do not contain a $K_{\ell+1}$ as a subgraph. This result was later generalized by Ne{\v{s}}et{\v{r}}il and R{\"o}dl~\cite{MR0412004} to the case of more than 2~colors. The smallest currently known graph that is $(K_3,2)$-Ramsey and $K_4$-free has 941 vertices \cite{MR2410116}.

Not allowing a $K_{\ell+1}$-subgraph is an entirely \emph{local} density restriction and still allows for graphs that are very dense \emph{globally}, in the sense that they contain many edges. Motivated by this fact, R{\"o}dl and Ruci{\'n}ski~\cite{MR1249720} asked how globally sparse Ramsey graphs can possibly be. They introduced the \emph{Ramsey density of $F$ and $r$}, defined as
\begin{equation} \label{eq:Ramsey-density}
  m^*(F,r):=\inf\{ m(G) \mid \text{$G$ is $(F,r)$-Ramsey} \} \enspace,
\end{equation}
where
\begin{equation} \label{eq:m-density}
  m(G):=\max_{H\seq G} \frac{e(H)}{v(H)} \enspace,
\end{equation}
and $e(H)$ and $v(H)$ denote the number of edges and vertices of $H$, respectively. The parameter $m(G)$ measures the global density of $G$; it is equal to half the average degree of $H$, maximized over all subgraphs $H\seq G$. This density parameter and variations of it arise naturally in the theory of random graphs \cite{MR1782847, MR1864966}, and also in Nash-Williams' theorem for the arboricity of a graph \cite{MR0161333} (this theorem actually plays a crucial role in our lower bound proofs later on; see also the remarks at the end of this paper).

Kurek and Ruci\'{n}ski~\cite{MR2152058} proved the somewhat surprising fact that the sparsest graph that is $(K_\ell,r)$-Ramsey (in the sense of \eqref{eq:Ramsey-density}), is a large complete graph on as many vertices as the Ramsey number $R(K_\ell,r)$ tells us; recall that the \emph{Ramsey number $R(F,r)$ of $F$ and $r$} is defined as the minimal $N=N(F,r)$ such that $K_N$ is $(F,r)$-Ramsey. Their result shows that the Ramsey density of cliques is
\begin{equation} \label{eq:mstar-cliques}
  m^*(K_\ell,r)=m(K_{R(K_\ell,r)})=\frac{R(K_\ell,r)-1}{2} \enspace.
\end{equation}
Apart from cliques, the only graphs for which the Ramsey density is known exactly are the trivial cases of stars $S_\ell$ with $\ell$ rays and $r\geq 2$ colors and the path~$P_3$ on 3~edges and $r=2$ colors: For stars any easy pigeonholing argument shows that
\begin{equation} \label{eq:mstar-stars}
  m^*(S_\ell,r)=m(S_{r(\ell-1)+1})=\frac{r(\ell-1)+1}{r(\ell-1)+2} \enspace.
\end{equation}
For $P_3$ we have $m^*(P_3,2)=1$, which is also not hard to see (for the upper bound proof consider the 5-cycle with an additional dangling edge attached to every vertex).
\TM{Rucinski behauptet in seinem Paper, dass $m^*(P_3,2)=1$ irgendwie aus dem Resultat zur star arboricity folgt. Aber ich sehe nicht, wie das geht.}

Also for an analogous parameter defined for \emph{vertex}-colorings, the so-called \emph{vertex-Ramsey density} introduced in~\cite{MR1249720} and further studied in \cite{MR1248487}, relatively little is known (even though one might suspect that vertex-colorings are much easier to deal with than edge-colorings). The authors of \cite{MR1248487} offered a prize money of 400,000 z{\l}oty (Polish currency in 1993) for the exact determination of the vertex-Ramsey density for the case where the forbidden graph~$F$ is the path on 3~vertices and $r=2$ colors are available.

\subsection{Our results}

In this work we determine the Ramsey density $m^*(F,r)$ up to some small error terms for several cases when $F$ is a complete bipartite graph, a cycle or a path, and $r\geq 2$ colors are available.

\subsubsection*{Complete bipartite graphs}

The first theorem summarizes our results for the case where $F$ is a complete bipartite graph $K_{a,b}$, $a\leq b$. In~\cite{kurek-thesis} a general upper bound of $m^*(K_{a,b},r)<r(a-1)+1$ has been derived. We are able to prove an almost matching lower bound for the case where $b$ is somewhat larger than $a$.

\begin{theorem}[Complete bipartite graphs]
\label{thm:bicliques}
For any integers $a\geq 2$, $b\geq (a-1)^2+1$ and $r\geq 2$ we have
\begin{equation}
\label{eq:mstar-bicliques}
  r(a-1)-\eps \leq m^*(K_{a,b},r) < r(a-1)+1 \enspace,
\end{equation}
where $\eps=\eps(a,b,r):=\frac{r(a-1)-1}{\max\{R(K_{a,b},r),2r(a-1)+1\}}<1/2$.
\end{theorem}

From the best known general lower bound $R(K_{a,b},r)\geq (2\pi\sqrt{ab})^{\frac{1}{a+b}}\big(\frac{a+b}{e^2}\big) r^{\frac{ab-1}{a+b}}$ from~\cite{MR0360329}, it follows that $\eps$ tends to $0$ for larger values of $a$, $b$ and/or $r$. See~\cite{MR0360329, radziszowski} for better lower bounds on $R(K_{a,b},r)$ in special cases that can be plugged into the lower bound in~\eqref{eq:mstar-bicliques}; see also the remarks at the end of this paper.

We note here that the upper bound for complete bipartite graphs $K_{a,b}$ stated in Theorem~\ref{thm:bicliques} (which holds for arbitrary values of $a$ and $b$) can be slightly improved; see the remarks at the end of this paper.

\subsubsection*{Cycles}

The next theorem summarizes our results for the case where $F$ is a cycle $C_\ell$. The upper bound for even cycles and the lower bound for odd cycles follow from results presented in~\cite{MR1249720}. For even cycles we are able to prove an almost matching lower bound.

\begin{theorem}[Cycles]
\label{thm:cycles}
For any even $\ell\geq 4$ and any integer $r\geq 2$ we have
\begin{equation} \label{eq:mstar-even-cycles}
  r-\eps \leq m^*(C_\ell,r) < r+1 \enspace,
\end{equation}
where $\eps=\eps(\ell,r):=\frac{r-1}{\max\{R(C_\ell,r),2r+1\}}<1/2$. \\
There is a function $f()$ such that for any odd $\ell\geq 3$ and any integer $r\geq 2$ we have
\begin{equation} \label{eq:mstar-odd-cycles}
  2^{r-1} \leq m^*(C_\ell,r) \leq f(r) \enspace.
\end{equation}
\end{theorem}

The dominant terms $r$ and $2^r$ in these bounds for even and odd cycles, respectively, are very similar to those known for the Ramsey number $R(C_\ell,r)$ (see \cite{MR0317991, MR1044995, JGT:JGT20572}). However, Theorem~\ref{thm:cycles} shows that, unlike the Ramsey number, the Ramsey density does not grow unbounded for fixed $r$ and $\ell\rightarrow\infty$.

Using the best known general lower bound $R(C_\ell,r)\geq (r-1)(\ell-2)+2$ from~\cite{MR2231999}, it follows that $\eps$ tends to $0$ for larger values of $\ell$ and/or $r$. See~\cite{MR2231999, radziszowski} for better lower bounds on $R(C_\ell,r)$ in special cases that can be plugged into the lower bound in~\eqref{eq:mstar-even-cycles}; see also the remarks at the end of this paper.

The existence of an upper bound for $m^*(C_\ell,r)$ which is independent of $\ell$, expressed by the function~$f(r)$ in~\eqref{eq:mstar-odd-cycles}, can be derived from the results in~\cite{MR1356576}. Unfortunately, due to the use of the regularity lemma those methods do not yield quantitative estimates for $f(r)$ that come close to the lower bound stated in \eqref{eq:mstar-odd-cycles}.
For fixed odd $\ell$ and $r\rightarrow \infty$, however, the trivial bound $m^*(C_\ell,r)\leq m(K_{R(C_\ell,r)})$ together with the bound $R(C_\ell,r)\leq (r+2)!\ell$ from \cite{MR0317991} shows that $m^*(C_\ell,r)\leq 2^{r\log_2(r)(1+o(1))}\ell$.

\subsubsection*{Paths}

We denote by $P_\ell$ the path on $\ell$ edges. Note that as $P_\ell$, $\ell\geq 3$, is a subgraph of $C_{2\ell}$, we have the upper bound
\begin{equation} \label{eq:Pell-Cell-bound}
  m^*(P_\ell,r)\leq m^*(C_{2\ell},r) \eqBy{eq:mstar-even-cycles} r+\Theta(1) \enspace.
\end{equation}
Our next theorem bounds away $m^*(P_\ell,r)$ from the Ramsey density of even cycles by showing that $m^*(P_\ell,r)\leq c\cdot r+\Theta(1)$ for some constant $c=c(\ell)$ that is strictly smaller than 1. For paths of length $\ell=3$ and $\ell=4$ we obtain almost matching bounds: $m^*(P_\ell,r)=\frac{1}{2}r+\Theta(1)$, i.e., here the truth is basically half the value on the right hand side of~\eqref{eq:Pell-Cell-bound}. (Note that paths of length $\ell=1$ and $\ell=2$ are already covered by \eqref{eq:mstar-stars}.)

\begin{theorem}[Paths]
\label{thm:paths}
For any integers $\ell\geq 3$ and $r\geq 2$ we have
\begin{equation} \label{eq:mstar-paths}
  \rhalf-\eps \leq m^*(P_\ell,r) < \left\lceil \big(1-{\textstyle \frac{1}{\ellhalf}}\big)\cdot r+{\textstyle \frac{1}{\ellhalf}} \right\rceil \enspace,
\end{equation}
where $\eps=\eps(\ell,r):=\frac{\rhalf-1}{\max\{R(P_\ellthird,r),2\rhalf+1\}}<1/2$.
\end{theorem}

We remark that the upper bound stated in Theorem~\ref{thm:paths} for $P_3$ has already been found in~\cite{kurek-thesis}.
See~\cite{radziszowski} for the best known lower bounds on $R(P_\ellthird,r)$ that can be plugged into the lower bound in \eqref{eq:mstar-paths}.

\subsection{Organization of this paper}

In Section~\ref{sec:bounds} we collect a few general bounds for the Ramsey density $m^*(F,r)$. We then prove Theorem~\ref{thm:cycles} in Section~\ref{sec:cycles}. We defer the proof of Theorem~\ref{thm:bicliques} to Section~\ref{sec:proof-bicliques}, as it reuses some of the ideas from the proof of Theorem~\ref{thm:cycles} and is somewhat more technical. In Section~\ref{sec:paths} we present the proof of Theorem~\ref{thm:paths}. Some concluding remarks and open problems are presented in Section~\ref{sec:remarks}.

\section{Some useful general bounds}
\label{sec:bounds}

We first collect several lower bounds for the Ramsey density $m^*(F,r)$ in terms of various graph parameters of $F$: the chromatic number $\chi(F)$, the 2-density $m_2(F):=\max_{H\seq F: v(H)\geq 3} \frac{e(H)-1}{v(H)-2}$, the minimum degrees $\delta(H)$ of subgraphs $H\seq F$ and the clique number $\omega(F)$.

\begin{lemma}[General lower bounds]
\label{lemma:general-lb}
For any graph $F$ and any $r\geq 2$ we have
\begin{enumerate}
\item $m^*(F,r) \geq \frac{(\chi(F)-1)^r}{2}$ \hspace{0.7mm} \cite{MR1249720};
\item $m^*(F,r) \geq \frac{r}{2} m_2(F)$, if $m_2(F)\geq 3$, \hspace{0.7mm} \cite{MR1249720, kurek-thesis};
\item $m^*(F,r) \geq \frac{1}{2}\Big(r\big(\max\limits_{H\seq F}\delta(H)-1\big)+1\Big)$ \hspace{0.7mm};
\item $m^*(F,r) \geq m^*(K_{\omega(F)},r) \eqBy{eq:mstar-cliques} \frac{R(K_{\omega(F)},r)-1}{2}$ \hspace{0.7mm}.
\end{enumerate}
\end{lemma}

\TM{Frage zur ersten Lower Bound: Gibt es eine Partition der Kanten von $K_9$ in 3 Graphen, sodass etwaige ungerade Kreise in jedem dieser Graphen Länge mindestens 7 haben? Falls, ja, so erhalten wir sofort eine bessere Lower Bound für $m^*(C_5,3)$.}

For each of the four lower bounds stated in Lemma~\ref{lemma:general-lb}, it is easy to find graphs $F$ for which this bound outmatches the other given bounds.
We further note that in the second lower bound stated in Lemma~\ref{lemma:general-lb}, the precondition $m_2(F)\geq 3$ can be relaxed to $m_2(F)>1$ in the case $r=2$ \cite{MR1249720} (see also \cite{kurek-thesis}).

The next lemma states an upper bound for the parameter $m^*(F,r)$ for the case where $F$ is bipartite. This lemma generalizes and in certain cases improves the upper bounds derived in~\cite{MR1249720} and \cite{kurek-thesis} for bipartite graphs.

\begin{lemma}[Upper bound for bipartite graphs]
\label{lemma:bipartite-ub}
For any bipartite graph $F=(A\dcup B,E)$ and any $r\geq 2$ we have
\begin{equation*}
  m^*(F,r) < r\big(d(F)-1\big)+1 \enspace,
\end{equation*}
where
\begin{equation} \label{eq:dF}
  d(F):=\min\{\Delta_A(F), \Delta_B(F)\} \enspace,
\end{equation}
and $\Delta_A(F)$ and $\Delta_B(F)$ denote the maximum degree of all vertices in $A$ and $B$, respectively.
\end{lemma}

Observe that for bipartite graphs~$F$ that satisfy the condition $\max_{H\seq F}\delta(H)=d(F)$, such as $k$-regular or complete bipartite graphs, the lower and upper bounds given by the third part of Lemma~\ref{lemma:general-lb} and by Lemma~\ref{lemma:bipartite-ub}, respectively, differ only by a factor of~2.

Note that for families of bipartite graphs $F$ for which the parameter $d(F)$ is constant (e.g.\ even cycles $C_\ell$, complete bipartite graphs $K_{2,\ell}$, $K_{3,\ell}$ etc.), the bound on $m^*(F,r)$ stated in Lemma~\ref{lemma:bipartite-ub} is \emph{independent} of the size of $F$.

Observe also that for large values of $r$ the upper bound given by Lemma~\ref{lemma:bipartite-ub} for bipartite graphs $F$ is \emph{much smaller} than the lower bound given by the first part of Lemma~\ref{lemma:general-lb} for non-bipartite graphs. In other words, depending on whether $F$ is bipartite or not we observe a dichotomy of the growth of the Ramsey density $m^*(F,r)$ in the number~$r$ of colors, a phenomenon very similar to what can be observed for the ordinary Ramsey numbers $R(F,r)$ (cf.~\cite{MR2520279, MR1044995}).

\TM{In~\cite{MR2520279} wird eine Upper Bound von $R(F,r)\leq 32 \Delta r^\Delta v(F)$ für bipartite Graphen $F$ mit Maximalgrad $\leq \Delta$ bewiesen, d.h.\ die Bound hängt polynomiell von $r$ ab. Demgegenüber liefert schon die Lower Bound von Chv{\'a}tal und Harary $R(F,r)\geq (\chi(F)-1)^{r-1}(v(F)-1)$ \cite{MR0314696} für nicht-bipartite Graphen eine exponentielle Abhängigkeit von $r$. In \cite{MR1044995} werden speziell $R(C_3,r)$ und $R(C_4,r)$ diskutiert, aber ich habe die Referenz mit hineingenommen, weil dort ein bisschen Intuition vermittelt wird.}

This dichotomous behavior of the Ramsey density does not occur for the already mentioned vertex-Ramsey density introduced in~\cite{MR1249720}, the analogous quantity in the \emph{vertex-}coloring setting --- for this parameter general bounds which differ only by a factor of~2 and which are based only on degree conditions on $F$ have been proven in~\cite{MR1248487}.

\subsection{Proof of Lemma~\texorpdfstring{\ref{lemma:general-lb}}{4}}

We only need to prove the last two bounds stated in Lemma~\ref{lemma:general-lb}.

\begin{proof}[Proof of (3) in Lemma~\ref{lemma:general-lb}]
Observe that for any graph $G=(V,E)$ and any integer $k$, if $m(G)<k/2$, then there is a vertex $v\in V$ with $\deg(v)\leq k-1$.

To prove the claimed lower bound fix a subgraph $H'\in\argmax_{H\seq F}\delta(H)$ and let $G$ be a graph with $m(G)<\frac{1}{2}\big(r (\delta(H')-1)+1\big)$. By the above observation we can order the vertices of $G$ from $v_1,\ldots,v_n$ such that for every $i=1,\ldots,n$ the degree of $v_i$ in $G[\{v_1,\ldots,v_i\}]$, the graph induced by the vertices $v_1,\ldots,v_i$, is at most $r(\delta(H')-1)$. For each $i=1,\ldots,n$, we color all edges incident to $v_i$ in $G[\{v_1,\ldots,v_i\}]$ by using each of the $r$ colors at most $\delta(H')-1$ times. This clearly yields a coloring of $G$ without a monochromatic copy of $H'$ and therefore without a monochromatic copy of $F$.
\end{proof}

\begin{proof}[Proof of (4) in Lemma~\ref{lemma:general-lb}]
This bound follows trivially from the observation that $F\seq F'$ implies that $m^*(F,r)\leq m^*(F',r)$.
\end{proof}

\subsection{Proof of Lemma~\texorpdfstring{\ref{lemma:bipartite-ub}}{5}}

For the proof we use a construction from~\cite{MR2664330}. In that paper, the authors construct Ramsey graphs with small minimum degree. As a priori a small minimum degree does \emph{not} imply sparseness (\wrt the $m$-density), it is somewhat surprising that the same construction also yields sparse Ramsey graphs.

For integers $n\geq 1$, $1\leq k\leq n$ and $m\geq 1$ we define a bipartite graph $G=G(n,k,m)=(N\dcup M,E)$ with vertex partition
\begin{subequations} \label{eq:G-n-k-m}
\begin{equation}
  N:=[n] \enspace, \quad M:={\textstyle \binom{[n]}{k}} \times [m]
\end{equation}
and edge set
\begin{equation}
  E:=\big\{ \{u,(S,v)\} \bigmid[\big] u\in S\in {\textstyle \binom{[n]}{k}} \wedge v\in [m] \big\}
\end{equation}
\end{subequations}
(here $[n]$ denotes the set $\{1,\ldots,n\}$ and $\binom{[n]}{k}$ the set of all $k$-element subsets of $[n]$).
In words, we construct $G$ by taking the vertex set $[n]$ and adding for each of the $\binom{n}{k}$ possible choices of $k$ different vertices from $[n]$ exactly $m$ many vertices that connect exactly to those vertices in $[n]$.

\begin{lemma}[$G$ is sparse]
\label{lem:density-G-n-k-m}
The graph $G=G(n,k,m)$ defined in \eqref{eq:G-n-k-m} satisfies $m(G)<k$.
\end{lemma}

\begin{proof}
Note that as the degree of all vertices in the set $M$ is exactly $k$, we have for any two nonempty subsets $A\seq N$ and $B\seq M$ that
\begin{equation*}
  \frac{e(G[A\cup B])}{v(G[A\cup B])} \leq \frac{k|B|}{|A|+|B|} < k \enspace,
\end{equation*}
proving the claim.
\end{proof}

The following lemma was proved in~\cite{MR2664330} (Lemma~2.6 in that paper; the proof there is stated only for $r=2$ colors, but generalizes straightforwardly to the general case).

\begin{lemma}[$G$ is Ramsey \cite{MR2664330}]
\label{lem:bipartite-construction}
~
\begin{itemize}
\item Let $F=(A\dcup B,E)$ be a bipartite graph and let $d(F)$ be defined as in~\eqref{eq:dF}. There are integers $n=n(F)$ and $m=m(F)$ such that $F$ is a subgraph of the graph $G(n,d(F),m)$ defined in~\eqref{eq:G-n-k-m}.
\item For any integers $n\geq 1$, $1\leq k\leq n$, $m\geq 1$ and $r\geq 2$ there are integers $n'=n'(n,k,r)$ and $m'=m'(k,m,r)$ such that $G(n',r(k-1)+1,m')$ is $(G(n,k,m),r)$-Ramsey.
\end{itemize}
\end{lemma}

\begin{proof}[Proof of Lemma~\ref{lemma:bipartite-ub}]
Let $F=(A\dcup B,E)$ be a bipartite graph and let $d(F)$ be defined as in~\eqref{eq:dF}. Combining the two parts of Lemma~\ref{lem:bipartite-construction} shows that there are integers $n'$ and $m'$ such that the graph $G(n',r(d(F)-1)+1,m')$ defined in~\eqref{eq:G-n-k-m} is $(F,r)$-Ramsey. By Lemma~\ref{lem:density-G-n-k-m} we have $m(G(n',r(d(F)-1)+1,m'))<r(d(F)-1)+1$, as required.
\end{proof}

\section{Proof of Theorem~\texorpdfstring{\ref{thm:cycles}}{2}}
\label{sec:cycles}

The upper bound for even cycles and the lower bound for odd cycles follow immediately from Lemma~\ref{lemma:bipartite-ub} and the first part of Lemma~\ref{lemma:general-lb}, respectively, so it remains to prove the claimed lower bound for even cycles and the claimed upper bound for odd cycles.

In fact, the claimed lower bound holds for cycles of arbitrary length, not just for even cycles, but is rather weak for odd cycles.

For the proof we will apply a well-known result of Nash-Williams on the arboricity of a graph~\cite{MR0161333}.
To state the result we define for any (multi)graph $G$
\begin{equation} \label{eq:m1-density}
  m_1(G):=\max_{H\seq G: v(H)\geq 2} \frac{e(H)}{v(H)-1}
\end{equation}
(cf.~\eqref{eq:m-density}).

\begin{theorem}[Nash-Williams' arboricity theorem \cite{MR0161333}]
\label{thm:nash-williams}
Let $r\geq 1$ be an integer. A loopless multigraph $G$ can be partitioned into at most $r$ forests if and only if $m_1(G)\leq r$.
\end{theorem}

\TM{Verallgemeinerung von Nash-Williams? Ein Graph $G$ kann genau dann in $r$ Graphen $G_1,\ldots,G_r$ mit $m_1(G)\leq c$ partitioniert werden, falls $m_1(G)\leq c\cdot r$ gilt. Für $c=1$ ist dies genau Theorem~\ref{thm:nash-williams}. Für $c=2,3,4,\ldots$ folgt dies unmittelbar durch mehrfache Anwendung von Theorem~\ref{thm:nash-williams}. Was ist für beliebige Werte $c\in\mathbb{R}$? Eine weitere Frage: Kann man $m_1()$ im Theorem (oder der Verallgemeinerung) durch $m()$ ersetzen?}

We next state an application of Theorem~\ref{thm:nash-williams}. To do so we define for any graph $G=(V,E)$ and any integer $k\geq 2$
\begin{equation} \label{eq:m1k-density}
  m_1(G,k):=\max_{H\seq G: v(H)\geq k} \frac{e(H)}{v(H)-1} \enspace.
\end{equation}
Note that we have
\begin{equation*}
  m_1(G)=m_1(G,2) \geq m_1(G,3) \geq m_1(G,4) \geq \cdots \enspace.
\end{equation*}

\begin{proposition}[Partition into $C_\ell$-free graphs]
\label{prop:nw-cycles}
Let $\ell\geq 3$ and $r\geq 1$ be integers. Any graph $G$ satisfying $m_1(G,R(C_\ell,r))\leq r$ can be partitioned into $r$ graphs which contain no $C_\ell$ as a subgraph.
\end{proposition}

Proposition~\ref{prop:nw-cycles} also holds if $G$ is a \emph{multigraph}, but we do not need this generalization here (but we do need the multigraph version of Theorem~\ref{thm:nash-williams} to prove Proposition~\ref{prop:nw-cycles}).

For the proof of Proposition~\ref{prop:nw-cycles} we need the following lemma.

\begin{lemma}[Contraction lemma]
\label{lem:contraction}
Let $k\geq 2$ and $r\geq 1$ be integers. For any graph $G$ with $m_1(G,k)\leq r$ there is a family $\cH$ of vertex-disjoint subgraphs of $G$ such that any $H\in\cH$ satisfies $v(H)<k$ and the multigraph $\Gbar$ obtained from $G$ by contracting every $H\in\cH$ into a single vertex satisfies $m_1(\Gbar)\leq r$.
\end{lemma}

\begin{proof}
We define a sequence of multigraphs $(G_i)_{i\geq 0}$ and a sequence $(\cH_i)_{i\geq 0}$, where $\cH_i$ is a family of subgraphs of $G_i$ as follows: Set $G_0:=G$, and for each $i\geq 0$ we define $\cH_i$ and $G_{i+1}$ inductively: We define $\tcH_i$ as the family of all subgraphs $H\seq G_i$ that satisfy $e(H)/(v(H)-1)>r$ and that are maximal with respect to this property (i.e., every proper supergraph $H'\supsetneq H$ of $G_i$ satisfies $e(H')/(v(H')-1)\leq r$). Then we let $\cH_i$ be a maximal subfamily of $\tcH_i$ with the property that any two different graphs $H,H'\in\cH_i$ are vertex-disjoint subgraphs of $G_i$. Let $G_{i+1}$ denote the multigraph obtained from $G_i$ by contracting every subgraph $H\in\cH_i$ into a single vertex (note that this may create multiple edges, but no loops).

\TM{Ein Beispiel für $r=2$ und einen Graphen, bei dem $\cH_i\subsetneq \tcH_i$ gilt: Ein $K_7$, ein $K_6$ und ein $K_7$, wobei jeder der $K_7$ mit dem $K_6$ über jeweils zwei Pfade bestehend aus 7 Kanten verbunden ist.}

We claim that for any $i\geq 1$ and any subgraph $H\in\cH_i$ of $G_i$, all vertices of $H$ are also vertices of $G_0=G$, i.e., none of the vertices of $H$ is obtained by contracting some $H'\in\cH_j$, $j<i$. To see this, suppose the claim was false, and consider the smallest $i\geq 1$ for which the claim was violated, i.e., consider a subgraph $H\in\cH_i$ of $G_i$ and a nonempty maximal set of graphs $\cH'\seq \bigcup_{0\leq j<i}\cH_i$ such that $H$ contains the vertices obtained from contracting each graph $H'\in\cH'$.
By the minimal choice of $i$, all graphs in $\cH'$ are vertex-disjoint subgraphs of $G_0=G$.
We clearly have
\begin{equation} \label{eq:H-H'-dense}
  \frac{e(H)}{v(H)-1}>r \qquad \text{and} \qquad \frac{e(H')}{v(H')-1}>r \quad \text{for all $H'\in\cH'$} \enspace.
\end{equation}
Let $0\leq j<i$ be the minimal integer for which a graph $J$ from $\cH_j$ is contained in $\cH'$. We show that the graph
\begin{equation*}
  J^+:=H\cup{} \bigcup_{H'\in\cH'}H' \enspace,
\end{equation*}
which is a proper supergraph of $J$, satisfies $e(J^+)/(v(J^+)-1)>r$, contradicting the maximal choice of $J$ in $G_j$. Note that we have
\begin{equation}
\begin{split} \label{eq:e-v-J+}
  e(J^+) &= e(H)+\sum_{H'\in\cH'} e(H') \enspace, \\
  v(J^+) &= v(H)+\sum_{H'\in\cH'} (v(H')-1) \enspace.
\end{split}
\end{equation}
Combining the above observations yields
\begin{equation*}
  \frac{e(J^+)}{v(J^+)-1} \eqBy{eq:e-v-J+} \frac{e(H)+\sum_{H'\in\cH'} e(H')}{v(H)-1+\sum_{H'\in\cH'} (v(H')-1)} \gBy{eq:H-H'-dense} r \enspace,
\end{equation*}
the desired contradiction.

By the above claim, \emph{all} graphs in $\cH:=\bigcup_{i\geq 0} \cH_i$ are vertex-disjoint subgraphs of $G_0=G$. It follows that by directly contracting every graph $H\in\cH$ into a single vertex we obtain a multigraph $\Gbar$ which by the definition of the graph sequence above satisfies $m_1(\Gbar)\leq r$ (otherwise we would have continued contracting subgraphs). Furthermore, it follows that every subgraph $H\in\cH$ of $G$ satisfies $e(H)/(v(H)-1)>r$, which by the assumption $m_1(G,k)\leq r$ implies that $v(H)<k$. This completes the proof.
\end{proof}

\begin{proof}[Proof of Proposition~\ref{prop:nw-cycles}]
Let $G$ be a graph with $m_1(G,R(C_\ell,r))\leq r$. By Lemma~\ref{lem:contraction} there is a family $\cH$ of vertex-disjoint subgraphs of $G$ such that any $H\in\cH$ satisfies $v(H)<R(C_\ell,r)$ and the multigraph $\Gbar$ obtained from $G$ by contracting every $H\in\cH$ into a single vertex satisfies $m_1(\Gbar)\leq r$.

Therefore, using that the subgraphs $H\in\cH$ of $G$ are vertex-disjoint and that $v(H)<R(C_\ell,r)$ holds for each of them, we can partition the edges contained in all those subgraphs into $r$ sets $E_1,\ldots,E_r$, such that none of those edge sets contains a $C_\ell$ as a subgraph.

Furthermore, using that $m_1(\Gbar)\leq r$ we can apply Theorem~\ref{thm:nash-williams} to partition the edges of $\Gbar$ into at most~$r$ forests. This clearly also yields a partition of the corresponding edges of $G$ into at most $r$~forests $F_1,\ldots,F_r$. It is easy to see that $E_1\cup F_1,\ldots,E_r\cup F_r$ is a $C_\ell$-free partition of the edges of~$G$, as desired.
\end{proof}

The next lemma shows that $m(G)$ is not much smaller than $m_1(G,k)$.

\begin{lemma}[Small $m$-density implies small $m_1$-density] \label{lem:min-m-G}
For any integers $k\geq 2$ and $r\geq 1$ and any graph $G$, if $m(G)<r-\frac{r-1}{\max\{k,2r+1\}}$ then we have $m_1(G,k)\leq r$.
\end{lemma}

\begin{proof}
We prove the contrapositive.
If $m_1(G,k)>r$ then by the definition in \eqref{eq:m1k-density} there is a subgraph $H\seq G$ satisfying
\begin{equation} \label{eq:eH-vH}
  \frac{e(H)}{v(H)-1}>r
\end{equation}
and
\begin{equation} \label{eq:vH1}
  v(H) \geq k \enspace.
\end{equation}
Clearly, we must have $e(H)\leq \binom{v(H)}{2}$, which combined with \eqref{eq:eH-vH} and using that $v(H)$ and $r$ are integers yields
\begin{equation} \label{eq:vH2}
  v(H) \geq 2r+1 \enspace.
\end{equation}
Again using that $e(H)$, $v(H)$ and $r$ are all integers, it follows from \eqref{eq:eH-vH} that
\begin{equation} \label{eq:eH-lb}
  e(H)\geq r(v(H)-1)+1=r v(H)-(r-1) \enspace.
\end{equation}
Combining the previous observations we obtain that
\begin{equation*}
  m(G)\geBy{eq:m-density} \frac{e(H)}{v(H)} \geBy{eq:eH-lb} r-\frac{r-1}{v(H)} \geByM{\eqref{eq:vH1},\eqref{eq:vH2}} r-\frac{r-1}{\max\{k,2r+1\}} \enspace,
\end{equation*}
as claimed.
\end{proof}

\begin{proof}[Proof of Theorem~\ref{thm:cycles}: lower bound for even cycles]
Let $G$ be a graph with $m(G)<r-\frac{r-1}{\max\{R(C_\ell,r),2r+1\}}$. By Lemma~\ref{lem:min-m-G} we have $m_1(G,R(C_\ell,r))\leq r$. Applying Proposition~\ref{prop:nw-cycles} shows that there is a coloring of the edges of $G$ with $r$ colors that avoids monochromatic copies of $C_\ell$.
\end{proof}

\begin{proof}[Proof of Theorem~\ref{thm:cycles}: upper bound for odd cycles]
In~\cite{MR1356576} Haxell, Kohayakawa and {\L}uczak defined for any integer $r\geq 2$ and for any sufficiently large integer $n\geq 1$ a graph $G=G(n,r)$ which (besides a number of other important properties) has a maximum degree that is bounded by a function depending only on~$r$ (Lemma~9 in~\cite{MR1356576}). It follows that also $m(G)$ is bounded by a function depending only on~$r$. The authors proved that for any fixed $r\geq 2$, the graph $G=G(n,r)$ has the property that, for any $r$-coloring of its edges, there is a color $s$ such that $G$ contains a monochromatic (induced) cycle $C_\ell$ in color $s$ for all $b\log n \leq \ell \leq b'n$, where $b=b(r)>0$ and $b'=b'(r)>0$ are functions depending only on~$r$ (Theorem~10 in~\cite{MR1356576}). Together these two results prove the existence of an upper bound on $m^*(C_\ell,r)$ that depends only on $r$.
\end{proof}

\section{Proof of Theorem~\texorpdfstring{\ref{thm:bicliques}}{1}}
\label{sec:proof-bicliques}

The upper bound follows immediately from Lemma~\ref{lemma:bipartite-ub}, so it remains to prove the claimed lower bound.

We will apply Nash-Williams' arboricity theorem (Theorem~\ref{thm:nash-williams}) to prove the following proposition.

\begin{proposition}[Partition into $K_{a,b}$-free graphs]
\label{prop:nw-bipartite}
Let $a\geq 2$, $b\geq (a-1)^2+1$ and $r\geq 1$ be integers. Any graph $G$ satisfying $m_1(G,R(K_{a,b},r))\leq r(a-1)$ can be partitioned into $r$ graphs which contain no $K_{a,b}$ as a subgraph.
\end{proposition}

As we shall see, the proof of Proposition~\ref{prop:nw-bipartite} is very similar to the proof of Proposition~\ref{prop:nw-cycles} presented in Section~\ref{sec:cycles}. The idea is to partition the edges of $G$ into $r(a-1)$ forests, and then group them into $r$ groups of size $a-1$ (this is where the proof differs from the cycle case where the $r$ forests are already the final partition). The next lemma shows that we do not create a copy of $K_{a,b}$ by taking the union of $a-1$ forests.

\begin{lemma}[$K_{a,b}$-free union of forests]
\label{lem:union-bipartite-free}
Let $a\geq 2$ and $b\geq (a-1)^2+1$ be integers. Furthermore, let $G$ be a graph and $\cH$ a family of vertex-disjoint subgraphs of $G$ such that the following conditions hold: The graphs $H\in\cH$ in $G$ are all $K_{a,b}$-free, and denoting by $\Gbar$ the multigraph obtained from contracting every subgraph $H\in\cH$ of $G$ into a single vertex, $\Gbar$ is the union of at most $a-1$ forests $F_1,\ldots,F_{a-1}$. Then $G$ contains no $K_{a,b}$ as a subgraph.
\end{lemma}

\TM{Man beachte, dass die Lower Bound $b\geq (a-1)^2+1$ bestmöglich ist.}

\begin{proof}
Suppose for the sake of contradiction that $G$ contains a $K_{a,b}$ as a subgraph. We denote by $A$ and $B$, $|A|=a$, $|B|=b$, the sets of vertices of $G$ corresponding to the two partition classes.

First observe that at most $a-1$ vertices from the set $A$ are contained in the same subgraph $H\in\cH$: Otherwise, as $H$ is $K_{a,b}$-free, at least one vertex $v$ from $B$ would not belong to $H$, and $a$ edges from $v$ would lead to vertices of $H$, i.e., there were $a$ parallel edges in $\Gbar$, contradicting the fact that $\Gbar$ is the union of at most $a-1$ forests.

We now define an auxiliary edge-colored multigraph $P$ on the vertex set $A$, where for each vertex $v\in B$, we add exactly one edge to $P$ as follows: We first consider the case that $v$ is contained in some graph $H\in\cH$ that also has some vertex $u$ in common with the set $A$. In this case, by the above observation there is a vertex $u'\in A$ which does not belong to $H$. Clearly, the edge $\{u',v\}$ belongs to some forest $F_i$, $i\in\{1,\ldots,a-1\}$. We then add the edge $\{u',u\}$ in color $i$ to $P$. The second case is that $v$ is not contained in a graph $H\in\cH$ which has vertices in common with the set $A$ (either because $v$ is not contained in \emph{any} $H\in\cH$ or because $v$ is contained in some $H\in\cH$, but $V(H)\cap A=\emptyset$). In this case each of the edges between $v$ and the vertices in $A$ belongs to one of the forests $F_1,\ldots,F_{a-1}$. By the pigeonhole principle at least two edges from the same forest $F_i$ lead to vertices in the set $A$. We pick two such vertices $u,u'\in A$ and add the edge $\{u,u'\}$ in color $i$ to $P$.

Observe that by construction of the multigraph $P$, we have that if $P$ contains a monochromatic cycle, then the edges of one of the graphs $F_i$, $i\in\{1,\ldots,a-1\}$, form a cycle in $\Gbar$.

As the number of vertices in $B$ is at least $(a-1)^2+1$, $P$ contains at least this many edges. By the pigeonhole principle at least $a$ of them have the same color. As $P$ has only $a$ vertices, it follows that those $a$ edges of the same color must form a cycle in $P$, implying that the edges of one of the graphs $F_i$, $i\in\{1,\ldots,a-1\}$, form a cycle in $\Gbar$, contradicting the fact that all those graphs are forests.
\end{proof}

\begin{proof}[Proof of Proposition~\ref{prop:nw-bipartite}]
Let $G$ be a graph with $m_1(G,R(K_{a,b},r))\leq r(a-1)$. By Lemma~\ref{lem:contraction} there is a family $\cH$ of vertex-disjoint subgraphs of $G$ such that any $H\in\cH$ satisfies $v(H)<R(K_{a,b},r)$ and the multigraph $\Gbar$ obtained from $G$ by contracting every $H\in\cH$ into a single vertex satisfies $m_1(\Gbar)\leq r(a-1)$.

Therefore, using that the subgraphs $H\in\cH$ of $G$ are vertex-disjoint and that $v(H)<R(K_{a,b},r)$ holds for each of them, we can partition the edges contained in all those subgraphs into $r$ sets $E_1,\ldots,E_r$, such that none of those edge sets contains a $K_{a,b}$ as a subgraph.

Furthermore, using that $m_1(\Gbar)\leq r(a-1)$ we can apply Theorem~\ref{thm:nash-williams} to partition the edges of $\Gbar$ into at most $r(a-1)$~forests. We denote the corresponding forests in $G$ by $F_{i,j}$, $1\leq i\leq r$, $1\leq j\leq a-1$. Applying Lemma~\ref{lem:union-bipartite-free} shows that $E_1\cup \bigcup_{j=1}^{a-1} F_{1,j},\ldots,E_r\cup \bigcup_{j=1}^{a-1} F_{r,j}$ is a $K_{a,b}$-free partition of the edges of~$G$, as desired.
\end{proof}

\begin{proof}[Proof of Theorem~\ref{thm:bicliques}: lower bound]
Let $G$ be a graph with $m(G)<r(a-1)-\frac{r(a-1)-1}{\max\{R(K_{a,b},r),2r(a-1)+1\}}$. By Lemma~\ref{lem:min-m-G} we have $m_1(G,R(K_{a,b},r))\leq r(a-1)$. Applying Proposition~\ref{prop:nw-bipartite} shows that there is a coloring of the edges of $G$ with $r$ colors that avoids monochromatic copies of $K_{a,b}$.
\end{proof}

\section{Proof of Theorem~\texorpdfstring{\ref{thm:paths}}{3}}
\label{sec:paths}

\subsection{Lower Bound}

We will apply Nash-Williams' arboricity theorem (Theorem~\ref{thm:nash-williams}) to prove the following proposition.

\begin{proposition}[Partition into $P_\ell$-free graphs]
\label{prop:nw-paths}
Let $\ell\geq 3$ and $r\geq 2$ be integers. Any graph $G$ satisfying $m_1(G,R(P_\ellthird,r))\leq \rhalf$ can be partitioned into $r$ graphs which contain no $P_\ell$ as a subgraph.
\end{proposition}

The proof of Proposition~\ref{prop:nw-paths} is very similar to the proof of Proposition~\ref{prop:nw-cycles} presented in Section~\ref{sec:cycles}. The idea is to first partition the edges of $G$ into $\rhalf$ forests, and then to split each of those forests again into two forests of stars.

\begin{proof}[Proof of Proposition~\ref{prop:nw-paths}]
Let $G$ be a graph with $m_1(G,R(P_\ellthird,r))\leq \rhalf$. By Lemma~\ref{lem:contraction} there is a family $\cH$ of vertex-disjoint subgraphs of $G$ such that any $H\in\cH$ satisfies $v(H)<R(P_\ellthird,r)$ and the multigraph $\Gbar$ obtained from $G$ by contracting every $H\in\cH$ into a single vertex satisfies $m_1(\Gbar)\leq \rhalf$.

Therefore, using that the subgraphs $H\in\cH$ of $G$ are vertex-disjoint and that $v(H)<R(P_\ellthird,r)$ holds for each of them, we can partition the edges contained in all those subgraphs into $r$ sets $E_1,\ldots,E_r$, such that none of those edge sets contains a $P_\ellthird$ as a subgraph.

Furthermore, using that $m_1(\Gbar)\leq \rhalf$ we can apply Theorem~\ref{thm:nash-williams} to partition the edges of $\Gbar$ into at most $\rhalf$~forests. By splitting each of those forests into two forests of stars in $\Gbar$, we obtain a partition of the corresponding edges of $G$ into star forests $F_1,\ldots,F_r$. We claim that $E_1\cup F_1,\ldots,E_r\cup F_r$ is a $P_\ell$-free partition of the edges of~$G$. To see this note that any path in $G$ within one of the edge sets $E_i\cup F_i$, $i\in\{1,\ldots,r\}$, can contain at most 2~edges from $F_i$ (the edges in $F_i$ form a star forest in $\Gbar$!). Those two edges connect at most 3 paths of length at most $\ellthird-1$ from $E_i$, showing that the total length of such a path is bounded by $2+3(\ellthird-1)=3\ellthird-1<\ell$.
\end{proof}

\begin{proof}[Proof of Theorem~\ref{thm:paths}: lower bound]
Let $G$ be a graph with $m(G)<\rhalf-\frac{\rhalf-1}{\max\{R(P_\ellthird,r),2\rhalf+1\}}$. By Lemma~\ref{lem:min-m-G} we have $m_1(G,R(P_\ellthird,r))\leq \rhalf$. Applying Proposition~\ref{prop:nw-paths} shows that there is a coloring of the edges of $G$ with $r$ colors that avoids monochromatic copies of $P_\ell$.
\end{proof}

\subsection{Upper Bound}

In this section we will construct a sparse $(P_\ell,r)$-Ramsey graph by using the bipartite graph $G=G(n,k,m)=(N\dcup M,E)$ defined in~\eqref{eq:G-n-k-m} as a building block. For any $k$-element subset $A\seq N$ we define $M(A):=A\times[m]\seq M$, and for any $B\seq M(A)$ we denote by $G[A\cup B]$ the (complete bipartite) subgraph of $G$ induced by the vertices in $A$ and $B$.

We call a coloring of the edges of a complete bipartite graph with vertex partition $A$ and $B$ an \emph{$A$-centered star coloring}, if each color class induces a star with its center at a vertex in $A$ and $|B|$ many rays. Note that in such a coloring, $|A|$ many different colors occur, and every vertex in $B$ is incident to edges in all those colors.

The next lemma states that any $r$-coloring of the edges of a large enough $G(n,k,m)$ that contains no monochromatic copies of $P_\ell$ must contain a star colored complete bipartite graph as a subgraph.

For integers $\ell\geq 3$, $s\geq 1$, $k\geq 2$ and $r\geq k$ we define
\begin{subequations} \label{eq:def-n-m}
\begin{align}
  n &= n(\ell,k,r):=R(\underbrace{P_\ellhalf,\ldots,P_\ellhalf}_{\text{$r$ times}},K_k) \enspace, \label{eq:def-n} \\
  m &= m(s,k,r):= \binom{r}{k}k!s \enspace, \label{eq:def-m}
\end{align}
\end{subequations}
where $R(G_1,\ldots,G_{r+1})$ denotes the generalized Ramsey number w.r.t.\ the graphs $G_1,\ldots,G_{r+1}$, i.e., the smallest integer $N=N(G_1,\ldots,G_{r+1})$ such that any $(r+1)$-coloring of the edges of $K_N$ contains a copy of $G_i$ in color~$i$ for some $i\in\{1,\ldots,r+1\}$.

\begin{lemma}[$G(n,k,m)$ contains a star colored subgraph]
\label{lem:star-subgraph}
Let $\ell\geq 3$, $s\geq 1$, $k\geq 2$ and $r\geq k$ be integers. Then for $n=n(\ell,k,r)$ and $m=m(s,k,r)$ as defined in~\eqref{eq:def-n-m} the graph $G=G(n,k,m)=(N\dcup M,E)$ defined in~\eqref{eq:G-n-k-m} has the following property: For \emph{any} $r$-coloring of its edges with no monochromatic copies of $P_\ell$ there is a $k$-element subset $A\seq N$ and an $s$-element subset $B\seq M(A)\seq M$ such that the coloring of the subgraph $G[A\cup B]$ is an $A$-centered star coloring.
\end{lemma}

The proof of Lemma~\ref{lem:star-subgraph} proceeds by repeatedly applying the pigeonhole principle.

\begin{proof}
We fix an $r$-coloring of the edges of $G$ with no monochromatic copies of $P_\ell$ and show that we can find the desired star colored subgraph.
For each $k$-element subset $A\seq N$ we call the subgraph $G[A\cup M(A)]$ \emph{colorful} if for each vertex in $M(A)$, all the $k$ edges incident to it have a different color (note that even if $G[A\cup M(A)]$ is colorful this coloring is not necessarily an $A$-centered star coloring).

The proof consists of two parts. In the first part we prove that there is at least one $k$-element subset $A\seq N$ such that $G[A\cup M(A)]$ is colorful. In the second part we prove that there is an $s$-element subset $B\seq M(A)\seq M$ such that the coloring on the subgraph $G[A\cup B]$ is an $A$-centered star coloring.

Suppose for the sake of contradiction that for every $k$-element subset $A\seq N$, the graph $G[A\cup M(A)]$ is \emph{not} colorful. Then we iteratively construct an auxiliary edge-colored complete graph $P$ on the vertex $N$ as follows: Initially, $P$ has no edges. As long as $P$ has an independent set of size $k$, we pick one such set $A$, and we pick a vertex $v\in M(A)$ in $G$ with two incident edges $\{a_1,v\}$ and $\{a_2,v\}$ of the same color (such a vertex exists as $G[A\cup M(A)]$ is not colorful), and we add the edge $\{a_1,a_2\}$ in this color to $P$. If $P$ has no independent set of size $k$ anymore, we add all the remaining non-edges to $P$ and assign them an additional $(r+1)$-st color. Note that by our construction, $P$ contains no clique of size $k$ in color $r+1$. By the definition in \eqref{eq:def-n} (recall that $|N|=n$), $P$ therefore must contain a monochromatic path of length $\ellhalf$ in one of the colors $1,\ldots,r$. But this path clearly corresponds to a monochromatic path of length $\ell$ in $G$, contradicting our assumption that the coloring of $G$ contains no monochromatic copies of $P_\ell$. This completes the first part of the proof.

For the second part we fix some $k$-element subset $A\seq N$ such that $G[A\cup M(A)]$ is colorful. Note that each vertex in $M(A)$ has $k$ edges in $k$ different colors incident to it. As the total number of colors is $r$, there are $\binom{r}{k}$ different possible color sets that can be incident to a vertex in $M(A)$. By the pigeonhole principle and the definition in~\eqref{eq:def-m}, it follows that there is a subset $B'\seq M(A)$ of size at least $k!s$ such that the colors incident to a vertex in $B'$ are the same for all vertices in $B'$. W.l.o.g.\ we assume that those are the colors $1,\ldots,k$. Now we focus on the vertices in the set $B'$. Fix some ordering of the vertices $a_1,\ldots,a_k$ in $A$, and for each vertex $b\in B'$ we consider the order in which the colors $1,\ldots,k$ appear on the edges $\{a_1,b\},\ldots,\{a_k,b\}$. Clearly, there are $k!$ different possible orders, and by the pigeonhole principle there must be a subset $B\seq B'$ of size at least $s$ such that the order of the colors is the same for all the vertices in $B$. It follows that the coloring on the subgraph $G[A\cup B]$ is an $A$-centered star coloring.
\end{proof}

We now define a huge graph $G^*=G^*(\ell,k,r)$ by repeatedly gluing together copies of the graph $G(n,k,m)$ defined in~\eqref{eq:G-n-k-m}. We later show that $G^*$ satisfies $m(G^*)<k$ and that this graph is $(P_\ell,r)$-Ramsey for a suitable choice of $k$ (see Lemma~\ref{lem:gstar-sparse} and Lemma~\ref{lem:gstar-Ramsey} below). To show the Ramsey property we will repeatedly apply Lemma~\ref{lem:star-subgraph} to find star colored subgraphs of $G^*$.
For the reader's convenience the following definition is illustrated in Figure~\ref{fig:gstar}.

\begin{figure}
\centering
\PSforPDF{
 \psfrag{n1}{$n_1$}
 \psfrag{n2}{$n_2$}
 \psfrag{s1}{$s_1$}
 \psfrag{sim1}[l][c][1][90]{$s_{i-1}=n_i$}
 \psfrag{si}[l][c][1][90]{$s_i=n_{i+1}$}
 \psfrag{sip1}[l][c][1][90]{$s_{i+1}=n_{i+2}$}
 \psfrag{sell}[l][c][1][90]{$s_\ellhalf=1$}
 \psfrag{m1}[l][c][1][90]{$m_1$}
 \psfrag{mim1}[l][c][1][90]{$m_{i-1}$}
 \psfrag{mi}[l][c][1][90]{$m_i$}
 \psfrag{mip1}[l][c][1][90]{$m_{i+1}$}
 \psfrag{mell}[l][c][1][90]{$m_\ellhalf$}
 \psfrag{k}{$k$} 
 \psfrag{aa}{$A$}
 \psfrag{mti}{$M\in T_i$}
 \psfrag{mip1a}{$M_{i+1}(A)$}
 \psfrag{g1}{$G(n_1,k,m_1)$}
 \psfrag{gi}[l][c][1][45]{$G(n_i,k,m_i)$}
 \psfrag{gip1}[l][c][1][45]{$G(n_{i+1},k,m_{i+1})$}
 \psfrag{gell}[l][c][1][45]{$G(n_\ellhalf,k,m_\ellhalf)$}
 \psfrag{gmi}{$G(m_{i-1},k,m_i)$}
 \psfrag{gmip1}{$G(m_i,k,m_{i+1})$}
 \psfrag{gmell}{$G(m_{\ellhalf-1},k,m_\ellhalf)$}
 \psfrag{ldots}{$\ldots$}
 \psfrag{vdots}{$\vdots$}
 \psfrag{gg1}{\Large $G_1$}
 \psfrag{ggi}{\Large $G_i$}
 \psfrag{ggip1}{\Large $G_{i+1}$}
 \psfrag{gstar}{\Large $G^*=G_\ellhalf$}
 \psfrag{a1}{$A_1=\{a_1,\ldots,a_k\}$}
 \psfrag{b1j}{$B_{1,j},\; 1\leq j\leq k+1$}
 \psfrag{b1}{$B_1=\bigcup\limits_{1\leq j\leq k+1} B_{1,j}$}
 \psfrag{bij}{$B_{i,j}$}
 \psfrag{bip1j}{$B_{i+1,j}$}
 \psfrag{aip1j}{$A_{i+1,j}$}
 \includegraphics{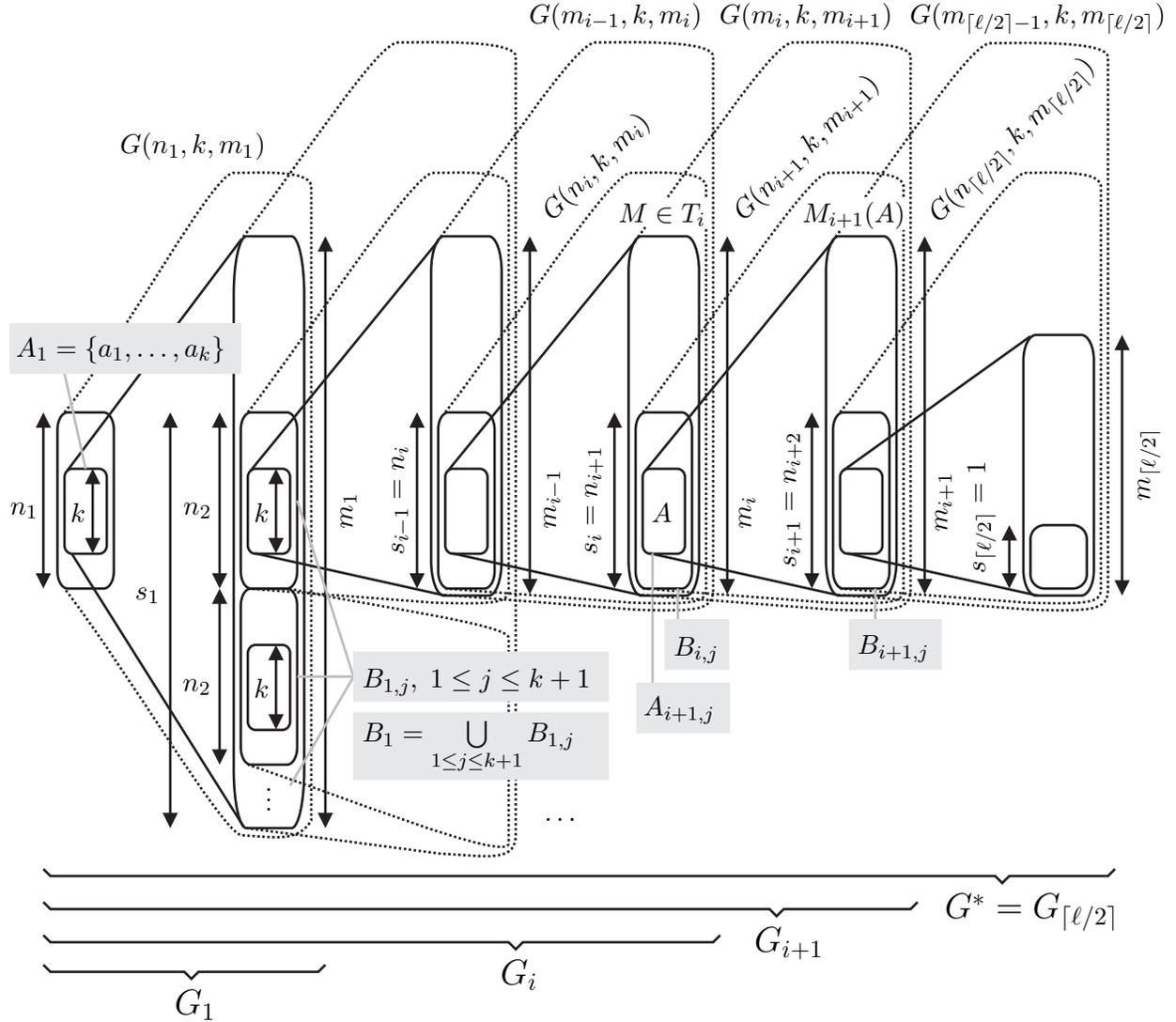}
}
\caption{Illustration of Definition~\ref{eq:def-gstar} and notations used in the proof of Lemma~\ref{lem:gstar-Ramsey}. The dotted lines represent copies of the graphs $G(n,k,m)$ for different values of $n$ and $m$. The variables in the grey boxes are only used in the proof of Lemma~\ref{lem:gstar-Ramsey}.} \label{fig:gstar}
\end{figure}

\begin{definition}
\label{eq:def-gstar}
Let $\ell\geq 3$, $k\geq 2$ and $r\geq k$ be fixed integers. In the following we define the graph $G^*=G^*(\ell,k,r)$. The definition proceeds in two steps.

We first define sequences of integers $n_i=n_i(\ell,k,r)$, $s_i=s_i(\ell,k,r)$ and $m_i=m_i(\ell,k,r)$ for $i=1,\ldots,\ellhalf$ as follows: Let $n_i:=n(\ell,k,r)$ for all $i=1,\ldots,\ellhalf$ where $n()$ is defined in~\eqref{eq:def-n} (so we have $n_1=\cdots=n_\ellhalf$). Furthermore, let $s_1:=(k+1)\cdot n_2$, $s_i:=n_{i+1}$ for $i=2,\ldots,\ellhalf-1$, and $s_\ellhalf:=1$. Then define $m_i:=m(s_i,k,r)$ for $i=1,\ldots,\ellhalf$ with $m()$ as defined in~\eqref{eq:def-m} (note that we have $s_1\geq s_2=\cdots=s_{\ellhalf-1}\geq s_\ellhalf$ and therefore $m_1\geq m_2=\cdots=m_{\ellhalf-1}\geq m_\ellhalf$).

Now we inductively define a sequence of graphs $G_1\seq \cdots\seq G_\ellhalf$ as follows:
First we set $G_1:=G(n_1,k,m_1)$ with $G(n_1,k,m_1)$ as defined in~\eqref{eq:G-n-k-m}. Similarly to before, for any $k$-element subset $A\seq [n_1]$ we denote by $M_1(A)$ the set of vertices in $G_1$ that are adjacent to all the vertices in $A$. Furthermore, we set $T_1:=\bigcup_{A\in\binom{[n_1]}{k}} \{M_1(A)\}$ ($T_1$ is a family of $m_1$-element sets).
For $i=1,\ldots,\ellhalf-1$ we construct $G_{i+1}$ from $G_i$ as follows: For every $M\in T_i$ ($M$ has size $m_i$) we glue a copy of $G(m_i,k,m_{i+1})$ onto the vertices in $M$ (such that the two vertex sets of size $m_i$ are identified). For any $k$-element subset $A$ of $M$ we define $M_{i+1}(A)$ as the set of vertices in this copy that are adjacent to all the vertices in $A$. Furthermore, we define $T_{i+1}:=\bigcup_{A\in\binom{M}{k} \wedge M\in T_i} \{M_{i+1}(A)\}$ ($T_{i+1}$ is a family of $m_{i+1}$-element sets).

Eventually we set $G^*:=G_\ellhalf$.
\end{definition}

\begin{lemma}[$G^*$ is sparse]
\label{lem:gstar-sparse}
The graph $G^*=G^*(\ell,k,r)$ from Definition~\ref{eq:def-gstar} satisfies $m(G^*)<k$.
\end{lemma}

For the proof of Lemma~\ref{lem:gstar-sparse} we use the following lemma, which follows immediately from Theorem~\ref{thm:nash-williams}, and which was used in similar form in~\cite{MR1194728}.

\begin{lemma}[Arboricity and orientations]
\label{lem:orientation}
Let $k\geq 1$ be an integer. The edges of a graph $G$ can be oriented acyclically such that the in-degree at each vertex is at most $k$ if and only if $m_1(G)\leq k$.
\end{lemma}

\begin{proof}[Proof of Lemma~\ref{lem:gstar-sparse}]
Note that we can orient the edges of $G^*$ acyclically such that the in-degree at each vertex is exactly $k$. Such an orientation can be found for each of the graphs $G_1\seq \cdots\seq G_\ellhalf=G^*$ from Definition~\ref{eq:def-gstar} by orienting all edges in $G_1$ away from the set $[n_1]$, and for each $i=1,\ldots,\ellhalf-1$ by orienting all edges in which $G_{i+1}$ and $G_i$ differ away from the vertices already present in $G_i$ (in Figure~\ref{fig:gstar}, this corresponds to orienting all edges from the left to the right). By Lemma~\ref{lem:orientation} we therefore have $m_1(G^*)\leq k$. Using the definitions in \eqref{eq:m-density} and \eqref{eq:m1-density} it follows that $m(G^*)<k$, as claimed.
\end{proof}

\begin{lemma}[$G^*$ is Ramsey]
\label{lem:gstar-Ramsey}
Let $\ell\geq 3$ and $r\geq 2$ be integers. For
\begin{equation} \label{eq:k}
  k:=\left\lceil \big(1-{\textstyle \frac{1}{\ellhalf}}\big)\cdot r+{\textstyle \frac{1}{\ellhalf}} \right\rceil
\end{equation}
the graph $G^*=G^*(\ell,k,r)$ from Definition~\ref{eq:def-gstar} is $(P_\ell,r)$-Ramsey.
\end{lemma}

\begin{proof}
For the reader's convenience, the notations used in the proof are illustrated in Figure~\ref{fig:gstar}.

We fix an $r$-coloring of the edges of $G^*$. Let $G_1\seq \cdots\seq G_\ellhalf=G^*$ and the subsets of vertices $M_i()$ and $T_i$, $i=1,\ldots,\ellhalf$, be as in Definition~\ref{eq:def-gstar}.

If the subgraph $G_1=G(n_1,k,m_1)$ of $G^*$ contains a monochromatic copy of $P_\ell$, we are done. Otherwise we apply Lemma~\ref{lem:star-subgraph} to this graph and obtain a $k$-element subset $A_1\seq [n_1]$ and an $s_1$-element subset $B_1\seq M_1(A_1)$ such that $G^*[A_1\cup B_1]$ is an $A_1$-star colored graph.
We denote the vertices in the set $A_1$ by $a_1,\ldots,a_k$ and we assume \wolog that all the edges in $G^*[A_1\cup B_1]$ incident to $a_s$ have color $s$ for $s=1,\ldots,k$.
We arbitrarily partition the set $B_1$ into $k+1$ sets $B_{1,1},\ldots,B_{1,k+1}$, each of size $n_2$ (recall the definition of $s_1$).

For $i=1,\ldots,\ellhalf-1$ and $j=1,\ldots,k+1$ we consider the subgraph of $G_{i+1}$ (and $G^*$) induced by the vertices in the sets $B_{i,j}$ and $\bigcup_{A\in \binom{B_{i,j}}{k}} M_{i+1}(A)$, which is clearly a copy of $G(n_{i+1},k,m_{i+1})$ ($B_{i,j}$ has size $n_{i+1}$). If this graph contains a monochromatic copy of $P_\ell$, we are done. Otherwise we apply Lemma~\ref{lem:star-subgraph} to this graph and obtain a $k$-element subset $A_{i+1,j}\seq B_{i,j}$ and an $s_{i+1}$-element subset $B_{i+1,j}\seq M_{i+1}(A_{i+1,j})$ (note here that $s_{i+1}=n_{i+2}$ for $i\leq \ellhalf-2$) such that $G^*[A_{i+1,j}\cup B_{i+1,j}]$ is an $A_{i+1,j}$-star colored graph.

For $j=1,\ldots,k+1$ we define a matrix $C^j\in\{0,1\}^{\ellhalf\times r}$ which encodes information about the colors of the edges of the subgraphs $G^*[A_1\cup B_{1,j}]$ and $G^*[A_{i,j}\cup B_{i,j}]$, $i=2,\ldots,\ellhalf$, as follows: We define $C^j_{i,s}:=1$ if and only if the color $s$ appears on the edges of the subgraph $G^*[A_{i,j}\cup B_{i,j}]$. Note that each row of $C^j$ contains exactly $k$ entries equal to 1, and that in the first row the first $k$ entries are equal to 1.

Observe that for each color $s=1,\ldots,k$, if the first $t$ entries of the $s$-th column of one of the matrices $C^j$ are all equal to 1, then the graph $G^*$ contains a path of length $t$ in color $s$ that starts at $a_s\in A_1$ and contains a vertex in each of the sets $A_{2,j},\ldots,A_{t+1,j}$ (in Figure~\ref{fig:gstar}, such a path goes from the left to the right).

The choice of $k$ in \eqref{eq:k} ensures that $r-(r-k)\ellhalf\geq 1$, implying that for each of the matrices $C^j$, $j=1,\ldots,k+1$, in one of the first $k$ columns all $\ellhalf$ entries are equal to 1. By the pigeonhole principle these all-one columns are the same for two of these matrices, implying that $G^*$ contains a monochromatic path of length $2\ellhalf\geq \ell$, as claimed.
\end{proof}

\begin{proof}[Proof of Theorem~\ref{thm:paths}: upper bound]
The proof follows immediately by combining Lemma~\ref{lem:gstar-sparse} and Lemma~\ref{lem:gstar-Ramsey}.
\end{proof}

\section{Concluding remarks and open questions}
\label{sec:remarks}

\begin{itemize}
\item Even though the results presented in this paper shed some light on the behavior of the Ramsey density for various interesting graph classes, some other graph classes are still very poorly understood. In particular, it would be very interesting to derive tight bounds for the Ramsey density of non-bipartite graphs $F$, specifically for odd cycles (cf.~Theorem~\ref{thm:cycles} and the first part of Lemma~\ref{lemma:general-lb}).

\item
For complete bipartite graphs $F=K_{a,b}$, $a\leq b$, a slightly better upper bound than the one stated in Theorem~\ref{thm:bicliques} can be derived from the results in~\cite{MR2285457} (see also \cite{kurek-thesis}): The authors show that the graph $K_{p,q}$ with $p:=r(a-1)+1$ and $q:=r(b-1)\binom{r(a-1)+1}{a}+1$ is $(K_{a,b},r)$-Ramsey. It follows that
\begin{equation} \label{eq:ub-Kab}
  m^*(K_{a,b},r) \leq m(K_{p,q})=\frac{pq}{p+q}=p-\frac{p^2}{p+q} \enspace,
\end{equation}
while Theorem~\ref{thm:bicliques} only yields an upper bound of $m^*(K_{a,b},r)< r(a-1)+1=p$. The difference $\frac{p^2}{p+q}$ between the two bounds is always less than 1, however.
For the special case $K_{2,2}=C_4$ and $r=2$ the best bounds we know are
\begin{equation*}
  \frac{11}{6} \leq m^*(C_4,2) \leq \frac{21}{10} \enspace,
\end{equation*}
where the upper bound follows from~\eqref{eq:ub-Kab} and the lower bound from~\eqref{eq:mstar-bicliques} (or alternatively, from \eqref{eq:mstar-even-cycles}) using that $R(K_{2,2},2)=R(C_4,2)=6$~\cite{MR0332559}.

\TM{Können wir zumindest für kleine Kreise (z.B.\ $C_4$ oder $C_5$, $r=2,3$) noch etwas mehr sagen (siehe Randbemerkung vorn zu $m^*(C_5,3)$)?}

\TM{Was können wir für den Fall $F=K_4^-$ sagen (der Diamond Graph), oder allgemeiner, für den Fall $K_\ell$ minus eine Kante? Können wir irgendetwas z.B.\ für den Fall $F=C_4$ daraus lernen?}

\item For any integer $d\geq 2$ and any graph $G$, define $a_d(G)$ as the minimum number of forests into which we can partition the edges of $G$ such that the components (=trees) of every forest in the partition have diameter at most $d$. For $d=\infty$, this is the well-known \emph{arboricity of $G$} \cite{MR0161333}, and for $d=2$ this is the so-called \emph{star arboricity} \cite{MR1001381, MR1194728}.
By Nash-Williams' theorem (Theorem~\ref{thm:nash-williams}), we have $a_\infty(G)=\lceil m_1(G)\rceil$ for any graph $G$. It is also not hard to see that
\begin{equation} \label{eq:ad-ub}
  a_d(G)\leq 2\cdot a_\infty(G)
\end{equation}
for any $d$ and any $G$. In our proof of the upper bound for $m^*(P_3,r)$ stated in Theorem~\ref{thm:paths} we exploited the fact that for $d=2$ and any integer $k\geq 1$ there is a graph $G$ satisfying
\begin{equation*}
  a_\infty(G)=k \quad \text{and} \quad a_2(G)=2k
\end{equation*}
(i.e., for these graphs the inequality in~\eqref{eq:ad-ub} is tight). Such graphs were first constructed in~\cite{MR1194728} and \cite{MR1206262}.
Our proofs show more generally that for any integers $d\geq 2$ and $k\geq 1$ there is a graph $G$ satisfying
\begin{equation} \label{eq:d-arboricity}
  a_\infty(G)=k
  \quad \text{and} \quad 
  a_d(G)=\left\lfloor \big(1+{\textstyle \frac{1}{\dhalf}}\big)\cdot k+1-{\textstyle \frac{1}{\dhalf}} \right\rfloor \enspace.
\end{equation}
Note that for $d=2$ and $d=3$ the right equation in \eqref{eq:d-arboricity} evaluates to $a_d(G)=2k$.
Can one find for some $d\geq 4$ and any $k\geq 1$ a graph $G$ satisfying $a_\infty(G)=k$ and $a_d(G)=2k$? If one could construct such graphs, which seems an interesting problem in its own right, then this would give an almost matching upper bound of $m^*(P_{d+1},r)\leq \big\lceil\frac{r+1}{2}\big\rceil$ (as we proved for $P_3$ and $P_4$, cf.~\eqref{eq:mstar-paths}).
On the other hand, if one could show that for some $d\geq 4$ and some large enough $k\geq 1$, any graph $G$ with $a_\infty(G)=k$ satisfies $a_d(G)\leq c\cdot k$ with a constant $c<2$, which again seems a challenging problem in itself, then this would immediately improve the lower bound on $m^*(P_{d+1},r)$ stated in~\eqref{eq:mstar-paths}.
\end{itemize}

\bibliographystyle{alpha}
\bibliography{refs}

\begin{thebibliography}{YYXB06}

\bibitem[AA89]{MR1001381}
I.~Algor and N.~Alon.
\newblock The star arboricity of graphs.
\newblock {\em Discrete Math.}, 75(1-3):11--22, 1989.
\newblock Graph theory and combinatorics (Cambridge, 1988).

\bibitem[AMR92]{MR1194728}
N.~Alon, C.~McDiarmid, and B.~Reed.
\newblock Star arboricity.
\newblock {\em Combinatorica}, 12(4):375--380, 1992.

\bibitem[BE73]{MR0317991}
J.~Bondy and P.~Erd{\H{o}}s.
\newblock Ramsey numbers for cycles in graphs.
\newblock {\em J. Combin. Theory Ser. B}, 14:46--54, 1973.

\bibitem[Bol01]{MR1864966}
B.~Bollob{\'a}s.
\newblock {\em Random graphs}, volume~73 of {\em Cambridge Studies in Advanced
  Mathematics}.
\newblock Cambridge University Press, Cambridge, second edition, 2001.

\bibitem[CG75]{MR0360329}
F.~Chung and R.~Graham.
\newblock On multicolor {R}amsey numbers for complete bipartite graphs.
\newblock {\em J. Combinatorial Theory Ser. B}, 18:164--169, 1975.

\bibitem[CH72]{MR0332559}
V.~Chv{\'a}tal and F.~Harary.
\newblock Generalized {R}amsey theory for graphs. {II}. {S}mall diagonal
  numbers.
\newblock {\em Proc. Amer. Math. Soc.}, 32:389--394, 1972.

\bibitem[DR08]{MR2410116}
A.~Dudek and V.~R{\"o}dl.
\newblock On the {F}olkman number {$f(2,3,4)$}.
\newblock {\em Experiment. Math.}, 17(1):63--67, 2008.

\bibitem[FL07]{MR2285457}
J.~Fox and K.~Lin.
\newblock The minimum degree of {R}amsey-minimal graphs.
\newblock {\em J. Graph Theory}, 54(2):167--177, 2007.

\bibitem[Fol70]{MR0268080}
J.~Folkman.
\newblock Graphs with monochromatic complete subgraphs in every edge coloring.
\newblock {\em SIAM J. Appl. Math.}, 18:19--24, 1970.

\bibitem[FS09]{MR2520279}
J.~Fox and B.~Sudakov.
\newblock Density theorems for bipartite graphs and related {R}amsey-type
  results.
\newblock {\em Combinatorica}, 29(2):153--196, 2009.

\bibitem[GRS90]{MR1044995}
R.~Graham, B.~Rothschild, and J.~Spencer.
\newblock {\em Ramsey theory}.
\newblock Wiley-Interscience Series in Discrete Mathematics and Optimization.
  John Wiley \& Sons Inc., New York, second edition, 1990.
\newblock A Wiley-Interscience Publication.

\bibitem[HK{\L}95]{MR1356576}
P.~Haxell, Y.~Kohayakawa, and T.~{\L}uczak.
\newblock The induced size-{R}amsey number of cycles.
\newblock {\em Combin. Probab. Comput.}, 4(3):217--239, 1995.

\bibitem[J{\L}R00]{MR1782847}
S.~Janson, T.~{\L}uczak, and A.~Rucinski.
\newblock {\em Random graphs}.
\newblock Wiley-Interscience Series in Discrete Mathematics and Optimization.
  Wiley-Interscience, New York, 2000.

\bibitem[KR94]{MR1248487}
A.~Kurek and A.~Ruci{\'n}ski.
\newblock Globally sparse vertex-{R}amsey graphs.
\newblock {\em J. Graph Theory}, 18(1):73--81, 1994.

\bibitem[KR05]{MR2152058}
A.~Kurek and A.~Ruci{\'n}ski.
\newblock Two variants of the size {R}amsey number.
\newblock {\em Discuss. Math. Graph Theory}, 25(1-2):141--149, 2005.

\bibitem[Kur92]{MR1206262}
A.~Kurek.
\newblock Arboricity and star arboricity of graphs.
\newblock In {\em Fourth {C}zechoslovakian {S}ymposium on {C}ombinatorics,
  {G}raphs and {C}omplexity ({P}rachatice, 1990)}, volume~51 of {\em Ann.
  Discrete Math.}, pages 171--173. North-Holland, Amsterdam, 1992.

\bibitem[Kur97]{kurek-thesis}
A.~Kurek.
\newblock {\em The density of Ramsey graphs}.
\newblock PhD thesis, AMU Pozna\'{n}, 1997.
\newblock In Polish.

\bibitem[{\L}SS11]{JGT:JGT20572}
T.~{\L}uczak, M.~Simonovits, and J.~Skokan.
\newblock On the multi-colored ramsey numbers of cycles.
\newblock {\em J. Graph Theory}, pages n/a--n/a, 2011.

\bibitem[NR76]{MR0412004}
J.~Ne{\v{s}}et{\v{r}}il and V.~R{\"o}dl.
\newblock The {R}amsey property for graphs with forbidden complete subgraphs.
\newblock {\em J. Combin. Theory Ser. B}, 20(3):243--249, 1976.

\bibitem[NW64]{MR0161333}
C.~Nash-Williams.
\newblock Decomposition of finite graphs into forests.
\newblock {\em J. London Math. Soc.}, 39:12, 1964.

\bibitem[Rad94]{radziszowski}
S.~Radziszowski.
\newblock Small {R}amsey numbers.
\newblock {\em Electron. J. Combin.}, 1994.
\newblock Dynamic survey, latest revision 2009.

\bibitem[Ram30]{Ramsey01011930}
F.~Ramsey.
\newblock On a problem of formal logic.
\newblock {\em Proc. London Math. Soc.}, s2-30(1):264--286, 1930.

\bibitem[RR93]{MR1249720}
V.~R{\"o}dl and A.~Ruci{\'n}ski.
\newblock Lower bounds on probability thresholds for {R}amsey properties.
\newblock In {\em Combinatorics, {P}aul {E}rd{\H o}s is eighty, {V}ol.\ 1},
  Bolyai Soc. Math. Stud., pages 317--346. J\'anos Bolyai Math. Soc., Budapest,
  1993.

\bibitem[SZZ10]{MR2664330}
T.~Szab{\'o}, P.~Zumstein, and S.~Z{\"u}rcher.
\newblock On the minimum degree of minimal {R}amsey graphs.
\newblock {\em J. Graph Theory}, 64(2):150--164, 2010.

\bibitem[YYXB06]{MR2231999}
S.~Yongqi, Y.~Yuansheng, F.~Xu, and L.~Bingxi.
\newblock New lower bounds on the multicolor {R}amsey numbers {$R_r(C_{2m})$}.
\newblock {\em Graphs Combin.}, 22(2):283--288, 2006.

\end{thebibliography}

\end{document}